\documentclass{scrartcl}
\usepackage[utf8]{inputenc}
\usepackage[T1]{fontenc}
\usepackage[UKenglish]{babel}
\usepackage{amsmath}
\usepackage{amsthm}
\usepackage{amssymb}
\usepackage{mathtools}
\usepackage{multirow}
\usepackage{enumerate}

\usepackage{todonotes}
\usepackage{hyperref}

\newcommand{\R}{\mathbb{R}}

\DeclareMathOperator{\Id}{Id}

\DeclareMathOperator{\sign}{sign}

\newtheorem{proposition}{Proposition}[section]
\newtheorem{corollary}[proposition]{Corollary}
\newtheorem{theorem}[proposition]{Theorem}
\newtheorem{lemma}[proposition]{Lemma}
\theoremstyle{remark}
\newtheorem{remark}[proposition]{Remark}

\newtheorem{example}[proposition]{Example}

\title{Observations on the metric projection in finite dimensional Banach~spaces}
\author{Christian Bargetz\footnote{Universität Innsbruck, Technikerstraße 13, 6020 Innsbruck, Austria, \texttt{christian.bargetz@uibk.ac.at}} \and Franz Luggin\footnote{Universität Innsbruck, Technikerstraße 13, 6020 Innsbruck, Austria, \texttt{franz.luggin@uibk.ac.at}}}

\begin{document}
\maketitle

\begin{abstract}
  \noindent\textbf{Abstract.} We consider the method of alternating (metric) projections for pairs of linear subspaces of finite dimensional Banach spaces. We investigate the size of the set of points for which this method converges to the metric projection onto the intersection of these subspaces. In addition we give a characterisation of the pairs of subspaces for which the alternating projection method converges to the projection onto the intersection for every initial point. We provide a characterisation of the linear subspaces of~$\ell_p^n$, $1<p<\infty$, $p\neq 2$, which admit a linear metric projection and use this characterisation to show that in~$\ell_p^3$, $1<p<\infty$, $p\neq 2$, the set of pairs of subspaces for which the alternating projection method converges to the projection onto the intersection is small in a probabilistic sense.\\[2mm]
  \noindent\textbf{Keywords.} Metric projection, finite dimensional Banach space, alternating projection method, linearity of the metric projection in~$\ell_p^n$\\[2mm]
  \vspace{3mm}
  \noindent\textbf{Mathematics Subject Classifications (2020).} 47J26, 41A65, 46B99
\end{abstract}

\section{Introduction}

For a Banach space $X$ and a closed subspace $M$, we denote by
\[
  P_{M}x := \{y\in M\colon \|x-y\|=d(x,M)\}
\]
the set of points which realise the distance between the subspace $M$ and the point $x\in X$. For strictly convex reflexive spaces the set $P_Mx$ is a singleton. Therefore, for these spaces and for every closed subspace $M$, we can consider the mapping
\[
  P_M\colon X\to X, \qquad x\mapsto P_{M}x
\]
which is called the \emph{metric projection onto~$M$}. Under some regularity assumptions on the Banach space $X$, which will always be satisfied in our setting, the metric projection is continuous, see e.g~\cite{HolmesKripke, Wulbert}.

Given closed linear subspaces $M, N\subset X$ we are interested in the question of whether the alternating projection method, which for $x_0 \in X$ is defined by
\[
  x_{2k+1} = P_M x_{2k} \qquad x_{2k} = P_Nx_{2k-1} \qquad \text{for}\; k\in\mathbb{N}_0,
\]
converges to the projection $P_{M\cap N} x_0$ of $x_0$ onto the intersection $M\cap N$.

In~\cite{vNe1949:RingsOfOperators}, J. von Neumann showed that if $X$ is a Hilbert space the alternating projection method always converges to the projection on to the intersection. A geometric proof of this result has been given by E.~Kopeck\'{a} and S.~Reich in~\cite{KopeckaReich}, see also~\cite{KopeckaReich2, KopeckaReich3}. However it turns out that this behaviour is strongly tied to the Hilbert space case. In~\cite{Sti1965:ClosestPointMaps} W.~Stiles showed that if for a Banach space $X$ of dimension at least three for all pairs of subspaces $(M,N)$ the alternating projection method converges to the projection onto the intersection $M\cap N$, then $X$ is a Hilbert space. In other words, every Banach space of dimension at least three, which is not a Hilbert space, contains at least two closed linear subspaces $M$ and $N$ for which the alternating projection method does not converge to the projection onto $M\cap N$. On the other hand, also in~\cite{Sti1965:ClosestPointMaps}, W.~Stiles showed that in finite dimensional smooth and strictly convex Banach spaces, the alternating projection method always converges to \emph{some} element of $M\cap N$. In~\cite{AtlestamSullivan}, B.~Atlestam and F.~Sullivan generalised these results to an infinite dimensional setting while imposing additional conditions on $M$ and $N$.

These results motivate the following three questions on the alternating projection method in finite dimensional Banach spaces.
\begin{enumerate}[(i)]
\item For subspaces $M$ and $N$ for which the alternating projection method does not converge to the projection onto $M\cap N$, how large is the set of points $x\in X$ for which $(P_MP_N)^nx$, or in the opposite order, converges to $P_{M\cap N}x$? One might hope that also in this case the set of points for which is occurs might be a large set.
\item For how many pairs of subspaces $(M,N)$ does the alternating projection method converge to the projection onto $M\cap N$ for all points of $X$?
\item How large is the distance between the limit of the alternating projection method and the projection onto the intersection?
\end{enumerate}

In this article, we address the questions~\emph{(i)} and~\emph{(ii)}.

In Section~2, we show that the convergence of the alternating projection method to the metric projection onto the intersection happens on a closed set $A$ containing $M\cap N$. We exhibit an example where both $A$ and its complement have nonempty interior and show that if $M\cap N$ contains an interior point of $A$, then $A=X$. We give a characterisation of pairs $(M,N)$ of subspaces where the alternating projection method converges to $P_{M\cap N}$ for all points and apply this to show that linearity of $P_{M\cap N}$ is sufficient for this behaviour.

In Section~3, we characterise the subspaces of $\ell_p^3$, $1<p<\infty$, $p\neq 2$ for which the metric projection is linear.

In Section~4 we apply the characterisation given in Section~3 to show that set of pairs~$(M,N)$ of subspaces of $\ell_p^3$ for which the alternating projection method converges to $P_{M\cap N}$ is small in a probabilistic sense.

The results of Sections~3 and~4 are part of the second author's master thesis~\cite{masterthesis} which was written under the supervision of the first author.
\section{Alternating projections in finite dimensional spaces}

We consider a finite-dimensional strictly convex and smooth Banach space~$X$ and two closed linear subspaces $M, N\subset X$.

We recall the following results due to W.~Stiles.
\begin{proposition}[{Theorem~3.1 in~\cite[p.~24]{Sti1965:ClosestPointMaps}}]\label{prop:kConstStiles}
  Let~$X$ be finite-dimensional, strictly convex, and smooth. Given two subspaces $M, N \subset X$ there is a number $0<k<1$ such that
  \[
    \|P_MP_Nx-P_Nx\| \leq k \|P_Nx-x\|
  \]
  for all $x\in M$.
\end{proposition}

W.~Stiles uses this result to show that for a strictly convex finite-dimensional Banach space $X$ the sequence $(P_MP_N)^nx$ converges to a point in $M\cap N$ for every $x\in X$ if and only if~$X$ is smooth. A careful examination of the proof given in~\cite{Sti1965:ClosestPointMaps} shows that for smooth spaces, the sequence $(P_MP_N)^n$ of operators converges uniformly on bounded sets. We summarise this observation in the following theorem.

\begin{theorem}\label{thm:Retraction}
  Let~$X$ be a finite-dimensional, strictly convex, and smooth Banach space and let $M,N\subset X$ two subspaces. The sequence $\{T_n\}_{n=1}^{\infty}$ where
  \[
    T_0=I, \qquad T_{2n+1}=P_NT_{2n} \qquad\text{and} \qquad T_{2n}=P_MT_{2n-1}
  \]
  converges to a retraction~$R_{M,N}$ onto $M\cap N$, uniformly on bounded sets. The retraction $R_{M,N}$ satisfies
  \[
    R_{M,N}(\lambda x) = \lambda R_{M,N}x \qquad \text{and}\qquad R_{M,N}(z+x) = z+ R_{M,N}x
  \]
  for all $z\in M\cap N$, $\lambda\in\mathbb{R}$ and $x\in X$.  
\end{theorem}

The following proof is based on parts of the proof of Theorem~3.2 in~\cite[pp.~25--26]{Sti1965:ClosestPointMaps} and is mainly included to keep the paper self-contained.

\begin{proof}
  By Proposition~\ref{prop:kConstStiles}, there is a number $0<k<1$
  \[
    \|P_MP_Ny-P_Ny\| \leq k \|P_Ny-y\|  \qquad\text{and}\qquad  \|P_NP_Mz-P_Mz\| \leq k \|P_Mz-z\|
  \]
  for all $y\in M$ and all $z\in N$. Let $R>0$ and $x\in B(0,R)$. From the above inequalities we obtain
  \[
    \|T_{n+1}x-T_nx\| \leq k \|T_nx-T_{n-1}x\|
  \]
  and hence for $m>n$ we have
  \begin{align*}
    \|T_nx-T_mx\| &\leq \|T_nx-T_{n+1}x\| + \ldots + \|T_{m-1}x-T_mx\|\\
                  &\leq k^{n-1} \|P_MP_Nx-P_Nx\| + \ldots + k^{m-2} \|P_MP_Nx-P_Nx\| \\
                  & \leq k^{n-1} \frac{1}{1-k} \|P_Nx\| \leq k^{n-1} \frac{2 R}{1-k}
  \end{align*}
  and hence $\{T_n\}_{n=1}^{\infty}$ is a Cauchy sequence for the topology of uniform convergence on bounded subsets of $X$. Hence it converges uniformly on bounded sets to a continuous operator~$R_{M,N}$. Since for $x\in X$ we have $T_{2n}x\in M$ and $T_{2n+1}x \in N$. Hence the limit has to be an element of both closed subspaces $M$ and $N$. Hence the range of $R_{M,N}$ has to be a subset of~$M\cap N$. Moreover, since $P_Mz=P_Nz=z$ for all $z\in M\cap N$, we have $R_{M,N}z=z$ for all $z\in M\cap N$ and hence $R_{M,N}$ is a continuous retraction onto $M\cap N$. The claimed properties of $R_{M,N}$ are a direct consequence of the definition of $R_{M,N}$ and the properties of the metric projections $P_M$ and $P_N$.
\end{proof}

\begin{remark}
  Note that the limit might depend on the order in which the projections with respect to $M$ and $N$ are taken. So the previous theorem implies the existence of two continuous retractions $R_{M,N}$ and $R_{N,M}$ onto $M\cap N$. They are connected via
  \[
    R_{M,N}P_N = R_{N,M} \qquad \text{and} \qquad  R_{N,M}P_M = R_{M,N}
  \]
  since
  \[
    R_{N,M} = \lim_{n\to\infty} (P_NP_M)^nx =  \lim_{n\to\infty} P_N(P_MP_N)^{n-1}P_Mx = P_NR_{M,N}P_Mx = R_{M,N}P_Mx
  \]
  and
  \[
    R_{M,N} = \lim_{n\to\infty} (P_MP_N)^nx =  \lim_{n\to\infty} P_M(P_NP_M)^{n-1}P_Nx = P_MR_{N,M}P_Nx = R_{N,M}P_Nx.
  \]
  In particular we have $R_{M,N}=R_{N,M}$ whenever these retractions commute with the projections $P_M$ and $P_N$.
\end{remark}

The following observation gives a criterion for the convergence of the alternating projection method to the projection onto the intersection.

\begin{proposition}\label{prop:ptwEquiv}
  For $x\in X$ the following assertions are equivalent.
  \begin{enumerate}[(i)]
  \item $R_{M,N}x = R_{N,M}x = P_{M\cap N}x$.
  \item $P_{M\cap N}(P_MP_N)^nx = P_{M\cap N}(P_NP_M)^n x = P_{M\cap N}x$ for all $n\in\mathbb{N}$.
  \end{enumerate}
\end{proposition}

\begin{proof}
  First note that if condition~\textit{(i)} is satisfied we have
  \[
    P_{M\cap N} x= \lim_{m\to\infty} (P_MP_N)^mx = \lim_{m\to\infty} (P_MP_N)^{m-n} (P_MP_N)^nx = P_{M\cap N}(P_MP_N)^nx
  \]
  and similarly $P_{M\cap N} x = P_{M\cap N}(P_NP_M)^nx$, i.e. condition~\textit{(ii)} holds.

  Conversely, note that
  \[
    R_{M,N}x = P_{M\cap N}R_{M,N}x = \lim_{n\to\infty} P_{M\cap N} (P_MP_N)^nx = \lim_{n\to\infty} P_{M\cap N} x = P_{M\cap N} x
  \]
  by continuity of $P_{M\cap N}$. The argument for $R_{N,M}=P_{M\cap N}$ is analogous.
\end{proof}

\begin{remark}
  Since both $R_{M,N}$ and $P_{M\cap N}$ are continuous, they agree on a closed subset of $X$ which contains $M\cap N$.
\end{remark}

In the following we want to see whether we can say more about the set $A$ of points for which the alternating projection algorithm converges to the projection onto the intersection. 

\begin{proposition}
  If there are $z\in M\cap N$ and $\varepsilon>0$ such that $R_{M,N}x=P_{M\cap N}x$ for all $x \in z + \varepsilon B_X$, then $R_{M,N}x = P_{M\cap N}x$ for all $x\in X$.
\end{proposition}

\begin{proof}
  Given $x\in X$ we set $y:=x-z$, $y' = \frac{\varepsilon}{2\|y\|} y$ and observe that
  \[
    y'\in \varepsilon B_X \qquad \text{and} \qquad x= z + \frac{2\|y\|}{\varepsilon} y'.
  \]
  Using Theorem~\ref{thm:Retraction}, we obtain that
  \begin{align*}
    R_{M,N}x &= R_{M,N}\left( z + \frac{2\|y\|}{\varepsilon} y'\right) = z + \frac{2\|y\|}{\varepsilon} R_{M,N} y'\\ & = \left(1-\frac{2\|y\|}{\varepsilon}\right) z + \frac{2\|y\|}{\varepsilon} R_{M,N}(z+y')= \left(1-\frac{2\|y\|}{\varepsilon}\right) z + \frac{2\|y\|}{\varepsilon} P_{M\cap N}(z+y')\\
    &= P_{M\cap N} (z+y) = P_{M\cap N}x,
  \end{align*}
  as claimed.
\end{proof}

\begin{remark}
  The previous results show that if the alternating projection method for two subspaces $M$ and $N$ does not converge to the metric projection onto the intersection $M\cap N$, there is no hope that it converges at least on a large set of points. Since the set where $R_{M,N}$ and $P_{M\cap N}$ agree is a closed set, convergence on a dense set already implies convergence everywhere. Moreover, by the previous result, if the set of points $x$ for which $R_{M,N}x=P_{M\cap N}x$ contains a ball intersecting $M\cap N$, it is already the whole space.

  Summing up, since the maps $P_{M\cap N}$ are homogeneous and additive with respect to elements of $M\cap N$ the set
  \[
    A = \{x\in X\colon R_{M,N}x=R_{N,M}x = P_{M\cap N} x \}
  \]
  is a closed set containing the subspace $M\cap N$ and the translates of a cone $C$ by every element of $M\cap N$. In other words the set $A$ is a closed set containing $(M\cap N) + C$ for the cone $C=(\ker P_{M\cap N}) \cap (\ker R_{M,N})\cap (\ker R_{N,M})$. 
\end{remark}

The next example shows that in general it is possible that both the sets $A$ above and its complement are somewhat large. 

\begin{example}
  We exhibit a simple example where both the set
  \[
    A = \{v\in X\colon R_{M,N}v=R_{N,M}v = P_{M\cap N} v \}
  \]
  and its complement have nonempty interior. Let $X$ be the space $\mathbb{R}^{3}$ equipped with the norm
  \[
    \|(x,y,z)\| =
    \begin{cases}
      \sqrt{x^2+y^2+z^2} & \text{if}\; xy \geq 0 \\
      \sqrt{\left(|x|^{3/2}+|y|^{3/2}\right)^{4/3} + z^2} & \text{if}\; xy < 0 \\
    \end{cases}
  \]
  which is obviously strictly convex on~$\mathbb{R}^{3}$. A direct computation shows that it is also continuously differentiable outside the origin, i.e. it is in particular a smooth norm. Moreover note $\|v\|_2\leq \|v\|$ for all $v\in\mathbb{R}^{3}$. In particular, if for a point $v\in X$ and its (Euclidean) orthogonal projection $u$ onto a subspace $L$ the first two entries of $v-u$ are non-negative, we have
  \[
    \|v-u\| = \|v-u\|_2 \leq \|v-w\|_2 \leq \|v-w\|
  \]
  for all $w\in L$ which implies that $P_Lv=u$. We consider the subspaces
  \[
    M = \mathrm{span}\left\{ \begin{pmatrix}1\\1\\1\end{pmatrix}, \begin{pmatrix}1\\0\\0\end{pmatrix} \right\} \qquad\text{and}\qquad N = \mathrm{span}\left\{ \begin{pmatrix}1\\1\\1\end{pmatrix}, \begin{pmatrix}0\\1\\0\end{pmatrix} \right\}.
  \]
  Note that the orthogonal projection of $v\in \mathbb{R}^{3}$ onto $M$ is given by
  \[
    w = v_1 \begin{pmatrix}1\\0\\0\end{pmatrix} + \frac{v_2+v_3}{2} \begin{pmatrix} 0\\1\\1 \end{pmatrix} = \begin{pmatrix} v_1\\ \frac{v_2+v_3}{2}\\\frac{v_2+v_3}{2}\end{pmatrix}
  \]
  and hence the set of points $v$ for which the first two entries of $v-w$ are non-negative contains the set
  \[
    Q:=\{(x,y,z)\in\mathbb{R}^{3}\colon x\geq 0, y\geq 0, z\leq 0\}.
  \]
  A similar computation shows that the same is true for~$N$. Since the projection onto a hyperplane is linear and $Q$ is three-dimensional, we may conclude that $P_M$ and $P_N$ coincides with the orthogonal projection onto $M$ and $N$, respectively. Moreover, since on $\mathbb{R}^{3}$ all norms are equivalent, von Neumann's results imply that $R_{M,N}=R_{N,M}=P$ where $P$ is the orthogonal projection onto the subspace $M\cap N$. A computation similar to the above one shows that on the set
  \[
    C:=\{(x,y,z)\in\mathbb{R}^{3}\colon x\geq 0, y\geq 0, z\leq -x-y\}.
  \]
  the metric projection onto $M\cap N$ coincides with the orthogonal projection. In particular on this cone, which has nonempty interior, the alternating algorithm converges to $P_{M\cap N}$. In order to show that also the set of points for which the alternating projection method does not converge to $P_{M\cap N}$ has nonempty interior, since this set is an open set, we only have to show that it is nonempty. For this aim consider $v=(-1,2,0)$ and note that the orthogonal projection is $R_{M,N}v= Pv= (1/3,1/3,1/3)$. Moreover, for the computation of the metric projection $P_{M\cap N}v$ note that for all $t\in(-1,2)$ the first entry of $v-t(1,1,1)$ is negative while the second one is positive. Observe that the derivative of $\|v-t(1,1,1)\|^2$ for $t\in(-1,1)$ is given as
  \[
    \frac{\mathrm{d}}{\mathrm{d}t} \|v-t(1,1,1)\|^2 = \frac{4}{3} \left(-\frac{3 (-t - 1)}{2 \sqrt{1+t}} - \frac{3 (2 - t)}{2\sqrt{2 - t}}\right) \left((1+t)^{3/2} + (2 - t)^{3/2}\right)^{1/3} + 2 t
  \]
  which does not vanish at $1/3$ but at a smaller value for $t$ which is $t_0 \approx 0.28$. Since for this $t_0$ we obtain $\|v-t_0(1,1,1)\|^2\approx 5.8$ which is smaller than both $\|v-(-1)(1,1,1)\|_2^2 =10$ and $\|v-2(1,1,1)\|_2^2 =13$, we conclude that $P_{M\cap N}v\neq R_{M,N}v$.  
\end{example}

We conclude this section with a characterisation of the pairs of subspaces $M$ and $N$ for which the alternating projection method converges for all initial points to the projection onto the intersection. For the next theorem, recall that, following~\cite{Sullivan}, a $B$-operator is a mapping $P\colon X\to X$ satisfying the following two conditions:
\begin{enumerate}[(i)]
\item $\|x-Px\|\leq \|x\|$ for all $x\in X$.
\item $\|x-Px\| = \|x\|$ if and only if $Px=0$.
\end{enumerate}

We can now give the following characterisation of when $R_{M,N} = R_{N,M} = P_{M\cap N}$.

\begin{theorem}\label{thm:EquivAltMethod}
  The following assertions are equivalent.
  \begin{enumerate}[(i)]
  \item $R_{M,N}=R_{N,M}=P_{M\cap N}$
  \item $P_{M\cap N} P_M = P_{M\cap N}$ and $P_{M\cap N} P_N = P_{M\cap N}$.
  \item $P_M(\ker P_{M\cap N}) \subset \ker P_{M\cap N}$ and $P_N(\ker P_{M\cap N}) \subset \ker P_{M\cap N}$
  \item $R_{M,N}$ and $R_{N,M}$ are $B$-operators.
  \end{enumerate}
\end{theorem}

\begin{proof}
  The equivalence of~\textit{(i)} and~\textit{(ii)} is a direct consequence of Proposition~\ref{prop:ptwEquiv}.

  That~\textit{(ii)} implies~\textit{(iii)} is obvious. For the converse implication note that
  \[
    P_{M\cap N}P_Mx - P_{M\cap N}x = P_{M\cap N}(P_Mx - P_{M\cap N}x) = P_{M\cap N}(P_M(x-P_{M\cap N}x)) = 0
  \]
  since $P_{M\cap N}x\in M\cap N \subset M$. Using a similar computation for $P_N$ finishes the proof of the implication \textit{(iii)}$\Rightarrow$\textit{(ii)}.

  The equivalence of~\textit{(i)} and~\textit{(iv)} is a direct consequence of Theorem~1 in~\cite[p.~248]{Sullivan}.
\end{proof}

\begin{corollary}
  If $P_{M\cap N}$ is linear, we have $R_{M,N} = R_{N,M} = P_{M\cap N}$.
\end{corollary}

\begin{proof}
  If $P_{M\cap N}$ is linear, the set $\ker P_{M\cap N}$ is a linear subspace. For $z\in \ker P_{M\cap N}$ we have
  \[
    P_M z = z + (P_Mz-z) \qquad\text{and}\qquad P_Nz = z + (P_Nz-z).
  \]
  Note that for $x\in X$ we have $P_Mx=0$ if and only if $\|x-y\|\geq \|x\|$ for all $y\in M$ and hence also for all $y\in M\cap N$ which is equivalent to $P_{M\cap N}x=0$. Hence $\ker P_M\subset \ker P_{M\cap N}$ and a similar argument shows that $\ker P_N\subset \ker P_{M\cap N}$.

  Since $P_Mz-z\in \ker P_M \subset \ker P_{M\cap N}$ and $P_Nz-z\in \ker P_N \subset \ker P_{M\cap N}$, we may deduce from the assertion that $\ker P_{M\cap N}$ is a linear subspace, that $P_Mz, P_Nz\in \ker P_{M\cap N}$. Now the claim follows from Theorem~\ref{thm:EquivAltMethod}.
\end{proof}

\section{Characterisation of the subspaces of $\ell_p^n$ with linear metric projection}
An extended version of these results, including more detailed proofs, can be found in the second author's master thesis~\cite{masterthesis}.

\begin{lemma}
  Let $X$ be a uniformly convex space and $M\subset X$ a closed subspace with metric projection~$P_M$. The mapping $Q:=\Id-P_M$ is a continuous retraction onto $\ker(P_M)$ which satisfies $\ker(Q)=M$.
\end{lemma}

\begin{proof}
  In order to see that $Q$ is a retraction onto $\ker(P_M)$ observe that
  \[
    QQx=(\Id-P_M)(x-P_Mx)=x-P_Mx-P_M(x-P_Mx)=x-P_Mx=Qx 
  \]
  and $Qx = x$ if and only if $P_Mx=0$. Since $P_Mx=x$ happens precisely for $x\in M$, we have $\ker(Q)=M$. 
\end{proof}

\begin{lemma}\label{lpn_kernel}
  Let $L\subset\ell_p^n$, $p>1$ and $n\geq 2$, be a one dimensional subspace spanned by the vector $a=(a_1,\ldots,a_n)$. Then, 
  \[
    \ker(P_L)= \Big\{x\in\R^n\colon \sum_{i=1}^na_i\sign(x_i)|x_i|^{p-1}=0\Big\}.
  \]
\end{lemma}

\begin{proof}
  By definition of the metric projection we have $P_L(x)=0$ if and only if
  \[
    \|x\|\leq \|x-\alpha a\| \qquad\text{for all}\quad  \alpha\in\mathbb{R}.
  \]
  Since the mapping $t\mapsto t^p$ is increasing for $t\geq 0$, this is equivalent to
  \[
    \|x\|^p \leq \|x-\alpha a\|^p \qquad\text{for all}\quad \alpha\in\mathbb{R}.
  \]
  Since the mapping defined by $f(\alpha)=\|x-\alpha a\|^p$ is a differentiable convex function, its minimum is characterised by $f'(\alpha)=0$. Computing this derivative results in the claimed characterisation.
\end{proof}

The following lemma is well-known, but since its proof is rather easy we include it for the convenience of the reader.

\begin{lemma}\label{lem:LinNullset}
  A metric projection $P$ onto a linear subspace $A$ of a uniformly convex Banach space $X$ is linear if and only if $\ker(P)$ is a linear subspace of $X$.
\end{lemma}

\begin{proof}
  We only have to show that $P$ is linear if $\ker(P)$ is a linear subspace. Assume that both $A$ and $\ker(P)=(\Id-P)[X]$ are linear subspaces of $X$. Given $x,y\in X$,$\lambda\in\R$, we observe that
  \[
    \lambda x+y=\lambda Px+Py+\lambda(\Id-P)x+(\Id-P)y=Pz_1+(\Id-P)z_2
  \]
  for some $z_1,z_2\in X$ and hence
  \[
    P(\lambda x+y) = P(Pz_1+z_2-Pz_2) = Pz_1+P(z_2)-Pz_2 = Pz_1=\lambda Px+Py,
  \]
  i.e. $P$ is a linear projection.
\end{proof}

\begin{lemma}\label{oplus_linear}
  Let $X$ be a uniformly convex and uniformly smooth Banach space and $A,B\subset X$ be linear subspaces where $A+B$ is closed. The metric projection $P_{A+ B}$ is linear if and only if $\ker(P_A)\cap \ker(P_B)$ is a linear subspace of $X$. In particular, if $P_A$ and $P_B$ are linear, $P_{A+ B}$ is linear.
\end{lemma}

\begin{proof}
  By the main theorem of~\cite[p.~117]{MR188759}, the sequence $((\Id-P_A)(\Id-P_B))^n$ of operators converges pointwise to the mapping $\Id-P_{A+ B}$. In particular, we obtain that $x=x-P_{A+B}x$ for all points $x\in \ker(P_A)\cap \ker(P_B)$. Hence,
  \[
    \ker(P_A)\cap \ker(P_B) \subset \ker(P_{A+B}).
  \]
  On the other hand, the condition $P_{A+ B}x=0$ is characterised by $\|x-(a+b)\|\geq\|x\|$ for all points $a\in A$ and $b\in B$. Since $0\in A\cap B$, we may conclude that 
  \[
    \|x-a\|\geqslant\|x\|\qquad\text{and}\qquad \|x-b\|\geqslant\|x\|
  \]
  for all $a\in A$ and all $b\in B$. In other words, we have $P_Ax=P_Bx=0$. Summing up, we have shown that $\ker(P_{A+B}) = \ker(P_A)\cap \ker(P_B)$. Now the claim follows from Lemma~\ref{lem:LinNullset}.
\end{proof}

\begin{theorem}\label{thm:MainThmLin}
  Let $L$ be a non-trivial subspace of $\ell_p^n=(\R^n,\|\cdot\|_p)$ with $p\in(1,\infty)\setminus\{2\}$, $n\geqslant2$. The metric projection $P_L$ is linear if and only if $L$ is of the form
  \[
    \bigoplus_{k=1}^d\R (e_{i_k}+\lambda_k e_{j_k}) \quad \text{with} \quad  d\in\{1,\dotsc,n\},\; i_k,j_k\in\{1,\dotsc,n\}\;\lambda_k\in\R.
  \]
  In other words the projection is linear if and only if the subspace is spanned by vectors with at most two nonzero entries.
\end{theorem}

For the proof we recall the following characterisation of the one-dimensional subspaces of $\ell_p$ with linear metric projection due to F.~Deutsch.

\begin{lemma}[{Corollary~5.3 in~\cite[p.~290]{Deutsch1982}}]\label{lem:DeutschOneDim}
  A one-dimensional subspace of $\ell_p$, $1<p<\infty$, $p\neq 2$, admits a linear metric projection if and only if it is spanned by an element with at most two non-zero coordinates.
\end{lemma}

The second tool we need for the proof of this theorem is the following lemma.

\begin{lemma}\label{lem:LinSys}
  Let $\phi\colon\mathbb{R}\to\mathbb{R}$ be a bijection, $\Phi\colon \R^n\to\R^n$ the function which applies $\phi$ to each entry $\Phi((x_1,\dotsc,x_n)):=(\phi(x_i))_{i=1}^n$, let $A\in\R^{d\times n}$ be a matrix with $d\leq n$ and let
  \[
    K=\{x\in\R^n\colon A\Phi(x)=0\}
  \]
  be a linear subspace of $\R^n$. Then there are row operations $E_1,\dotsc,E_m\in \R^{d\times d}$ such that the solution set
  \[
    \widetilde K_i=\{x\in\R^n\colon \widetilde A_{i-}\Phi(x)=0\}
  \]
  of every single row of our new matrix $\widetilde A=E_m\dotsm E_1A$ is a linear subspace. 
\end{lemma}

\begin{proof}
  We show this by induction over the dimension~$d$. Since the case $d=1$ is obvious, we assume that $d>1$. We can eliminate entries from $A$ using row operations until every row in $A$ has a pivot. However, since we cannot permute columns, in general, the pivot of the $i$-th row will not be at the $i$-th position, so let $j(i)$ denote the column index of the $i$-th pivot.
  Now our equation looks like this:
  \begin{gather*}
    \left[
      \begin{array}{c|c|c|c|c|c|c}
        \multirow{5}{*}{\huge0}\ &a_{1,j(1)}&\ *\ &0&*&0&\ \multirow{6}{*}{\huge*}\\
                                 &0&0&a_{2,j(2)}&*&\vdots&\\
                                 &\vdots&\vdots&\ddots&\ddots&0&\\
                                 &0&0&0&0&a_{d,j(d)}&
      \end{array}\right]\Phi(x)=0.
  \end{gather*}
  where the pivots are the only non-zero entries in their respective column and the rows are ordered in such a way that $j$ is strictly increasing. Let $P$ be the invertible matrix that transforms $A$ into row echelon form, i.e. $PA$ is the matrix above.

  Note that without loss of generality, we may assume that there are no leading columns of zeros. This allows us also to assume without loss of generality that $j(1)=1$. Then we can define
  \[
    f_1(x_2,\dotsc,x_n):=\phi^{-1}\left(-\frac{\vphantom{\frac11}\sum_{j=2}^n(PA)_{1j}\phi(x_j)}{(PA)_{11}}\right)
  \]
  and $M_1:=\bigcap_{i=2}^d\tilde{K}_i$. We get that $\tilde{K}_1=\{x\in\R^n\colon x_1=f_1(x_2,\dotsc,x_n)\}$ and
  \[
    K=\left\{\left(\vphantom{\sum}f_1(x_2,\dotsc,x_n),x_{2},\dotsc,x_n\right)\in\R^n \colon (x_2,\dotsc,x_n)\in M_1\right\}=\Gamma_{f_1|_{M_1}}.
  \]
  Since $K$ is a linear subspace, the above implies that the graph of $f_1|_{M_1}$ is a linear subspace and hence $f_1|_{M_1}$ is linear which implies that $M_1$ is a linear subspace. In other words, we have shown that the intersection $\bigcap_{i=2}^{d}\tilde{K}_i$ is a linear subspace. We can now repeat this process for the $(d-1)\times n$ matrix
  \[
    A_1:=\left(\left((PA)_{ij}\right)_{i=2}^d\right)_{j=1}^n
  \]
  that consists of all entries of $PA$ except for the first row. We can do this since we now know that $M_1$, the set defined by $A_1\tilde\Phi(x_2,\dotsc,x_n)=0$, is a linear subspace. Moreover, the first column of $A_1$ is zero and can therefore be omitted. Now the claim follows by the induction hypothesis.
\end{proof}

\begin{proof}[Proof of Theorem~\ref{thm:MainThmLin}]
  That the condition on~$L$ is sufficient is a direct consequence of Lemma~\ref{lem:DeutschOneDim} and Lemma~\ref{oplus_linear}.

  For the proof of the converse implication, we use an inductive argument over the dimension of the subspace $L\subset\ell_p^n$ and note that the one-dimensional case follows from Lemma~\ref{lem:DeutschOneDim}. We are left to consider the case of a $d$-dimensional subspace represented as $L=\bigoplus_{i=1}^d L_i$ where all $L_i$ are one-dimensional. Since we assume that $P_L$ is linear, we may use Lemma~\ref{oplus_linear} to conclude that the set $\bigcap_{i=1}^{d}\ker(P_{L_i})$ is a linear subspace. We now show that this implies the existence of a representation $L = \bigoplus_{i=1}^{d} \tilde{L}_i$ where each of the $\tilde{L}_i$ is spanned by an element with at most two nonzero entries. To this end, we pick a basis of the $L_i$. So, let $a^{(i)}\in L_i\setminus\{0\}$ and write their coordinates with respect to the standard basis into a $d\times n$ matrix $A:=({a_j^{(i)}})_{ij}$. Then by Lemma~\ref{oplus_linear} the intersection $K=\ker(P_{L_1})\cap \cdots\cap \ker(P_{L_d})$ can be described as
  \[
    K= \left\{x\in\R^n \colon A \begin{bmatrix}\sign(x_1)|x_1|^{p-1} \\ \vdots \\ \sign(x_n)|x_n|^{p-1} \end{bmatrix} =0\right\}.
  \]
  Since $\phi\colon\R\rightarrow\R\colon x\mapsto \sign(x)|x|^{p-1}$ is bijective for all $p>1$, we can apply elementary row operations to the system of linear equations $Ay=0$ where $y_i=\phi(x_i)$ without changing its solution set, and due to the one-to-one correspondence between $y$ and $x$ that means the solution set of $A\Phi(x)=0$ also stays the same under row operations (here, $\Phi(x)$ is a shorthand for $[\Phi(x_j)]_{j=1}^n$). Due to this fact, we will look at $A\Phi(x)=0$ as if it were a system of linear equations even if it is not when viewed as equations in $x$. Now we may apply Lemma~\ref{lem:LinSys} to obtain a matrix $\tilde{A}$ such that $\tilde{A}\Phi(x)=0$ if and only if $A\Phi(x)=0$ and with the property that the sets $\tilde{K}_i = \{x\in\mathbb{R}^{n} \colon \tilde{A}_{i-}\Phi(x)=0\}$ are linear subspaces. By Lemma~\ref{lpn_kernel} the sets $\tilde{K}_i$ are kernels of the metric projection onto a single one-dimensional subspace $\widetilde L_i$. Since they are linear subspaces, Lemmas~\ref{lem:LinNullset} and~\ref{lem:DeutschOneDim} imply that each of them is spanned by an element with at most two nonzero entries.
\end{proof}

\section{Alternating projections and random subspaces of $\ell_p^3$}
Since by Theorem~2.3 in~\cite[p.~22]{Sti1965:ClosestPointMaps} every at least three-dimensional Banach space $X$ which is not isomorphic to an inner product space, has to contain two subspaces $N$ and $M$ and a point $x\in X$ such that the sequence of iterates $(P_MP_N)^nx$ does not converge to $P_{M\cap N}x$, it seems natural to investigate the size of the set of pairs of subspaces $(M,N)$  where the alternating projection method converges to a projection onto the intersection. We will address this question in the particular case of $\ell_p^3$. Since in a three-dimensional space, the only nontrivial case is the one of two two-dimensional subspaces, we restrict ourselves to the case of pairs of two-dimensional subspaces of $\ell_p^3$.

We start our investigation by providing a description of the kernel of the metric projection onto a linear subspace of a finite dimensional $\ell_p$-space using the duality mapping.
\begin{lemma}\label{dualityofAorthogonal}
  The kernel of $P_A\colon\ell_p^n\rightarrow\ell_p^n$ can be written as $\ker P_A=j_q(A^\perp)$ using the duality mapping
  \[
    j_q(x)=\begin{bmatrix}
      \sign(x_1)|x_1|^{q-1}\\
      \vdots\\
      \sign(x_n)|x_n|^{q-1}
    \end{bmatrix},
  \]
  $q:=\frac p{p-1}$ the Hölder complement of $p$ and $A^\perp$ the $\ell_2$-orthogonal complement of $A$.
\end{lemma}

\begin{proof}
  From Lemma~\ref{lpn_kernel} we deduce that 
  \[
    \ker P_A =\left\{x\in\R^3\ \middle|\ \sum_{i=1}^na_i^{(k)}\sign(x_i)|x_i|^{p-1}=0\;\text{for}\;k=1,\ldots,m\right\}
  \]
  for some basis $\{a^{(1)},\ldots,a^{(m)}\}$ of $A$. These conditions are equivalent to the $\ell_2$-orthogonality of both $a^{(1)}, \ldots, a^{(m)}$ to the vector
  $j_p(x)$. Since $j_p^{-1}=j_q$, see e.g. Corollary~3.5 in~\cite[p.~62]{MR1079061}, we obtain
  \begin{gather*}
    \ker(P_A)=\left\{x\in\R^3\ \middle|\ \langle a^{(1)},j_p(x)\rangle= \ldots=\langle a^{(m)},j_p(x)\rangle=0\right\}=\\
    =\left\{j_q(x)\in\R^3\ \middle|\ \langle a^{(1)},x\rangle=\ldots=\langle a^{(m)},x\rangle=0\right\}=j_q(A^\perp),
  \end{gather*}
  which finishes the proof.
\end{proof}

We call a two-dimensional subspace $A$ of $\R^3$ \emph{chosen uniformly at random} if there is a vector $a$ chosen uniformly from the Euclidean unit sphere such that $\langle a,x\rangle=0$ for all $x\in A$, i.e. $A$ is the $\ell_2$-\!\! orthogonal complement of some uniformly random $\ell_2$-unit vector~$a$.
\begin{proposition}\label{probzero}
Let $X=\ell_p^3$ and $A$ and $B$ be two-dimensional subspaces chosen uniformly at random. Then $P_A$ and $P_B$ are linear, but $\mathbb P(P_{A\cap B}\text{ is linear})=0$.
\end{proposition}

\begin{proof}
  Since the metric projection onto hyperplanes is always linear, in~$\R^3$ all planes have a linear projection, i.e. the maps $P_A$ and $P_B$ are linear. On the other hand, by Theorem~\ref{thm:MainThmLin} $P_{A\cap B}$ is linear iff $A\cap B$, which is almost surely one-dimensional, is spanned by a vector with at most two non-zero entries.
  Since the space $A\cap B$ is the set of all points which are $\ell_2$-orthogonal to both $a$ and $b$, we know that
  \[
    A\cap B=\R(a\times b)=\R\begin{bmatrix}a_2b_3-a_3b_2\\
      a_3b_1-a_1b_3\\
      a_1b_2-a_2b_1\end{bmatrix}.
  \]
  Therefore using the union bound and the fact that $a_i\sim a_j\sim b_k$ for any $i,j,k\in\{1,2,3\}$ we conclude that
  \begin{align*}
    \mathbb P(P_{A\cap B}\text{ is linear})&=\mathbb P(a_2b_3=a_3b_2\lor a_3b_1=a_1b_3\lor a_1b_2=a_2b_1)\\\leq3\mathbb P(a_1b_2=a_2b_1).    
  \end{align*}
  We are left to compute the latter probability. By construction of our random model, we have
  \[
    a=\begin{bmatrix}
      \cos\alpha\sqrt{1-\phi^2}\\
      \sin\alpha\sqrt{1-\phi^2}\\
      \phi
    \end{bmatrix},\ 
    b=\begin{bmatrix}
      \cos\beta\sqrt{1-\psi^2}\\
      \sin\beta\sqrt{1-\psi^2}\\
      \psi
    \end{bmatrix}, \alpha,\beta\in\mathcal{U}_{[-\pi,\pi]},\phi,\psi\in\mathcal{U}_{[-1,1]}.
  \]
  Therefore,
  \begin{align*}
    \mathbb P(a_1b_2=a_2b_1)& =\mathbb P\left(\cos\alpha\sqrt{1-\phi^2}\sin\beta\sqrt{1-\psi^2}=\sin\alpha\sqrt{1-\phi^2}\cos\beta\sqrt{1-\psi^2}\right)\\
                            &\leq\mathbb P(\phi^2=1)+\mathbb P(\psi^2=1)+\mathbb P(\sin(\beta-\alpha)=0)\\
                            &=0+\mathbb P(\{k\in\mathbb Z\colon\beta=\alpha+ k\pi\})=0,
  \end{align*}
  as claimed.
\end{proof}

\bigskip
{\noindent\textbf{\textsf{Acknowledgement.}}} This research was funded by the Tyrolean Science Fund (Tiroler Wissenschaftsförderung).

\end{document}